\documentclass[amscd,amssymb,verbatim,11pt]{amsart}

\begin{document}

\title[$k$-central binomial coefficients]{The $p$-adic valuation 
of $k$-central binomial coefficients}

\author{Armin Straub}
\address{Department of Mathematics,
Tulane University, New Orleans, LA 70118}
\email{astraub@math.tulane.edu}

\author{Victor H. Moll}
\address{Department of Mathematics,
Tulane University, New Orleans, LA 70118}
\email{vhm@math.tulane.edu}

\author{Tewodros Amdeberhan}
\address{Department of Mathematics,
Tulane University, New Orleans, LA 70118}
\email{tamdeberhan@math.tulane.edu}


\date{\today}


\begin{abstract}
The coefficients $c(n,k)$ defined by 
$$(1-k^{2}x)^{-1/k} = \sum_{n \geq 0} 
c(n,k)x^n$$ 
\noindent
reduce to the central binomial coefficients $\binom{2n}{n}$ 
for $k=2$. Motivated by a question of H. Montgomery and H. Shapiro for 
the case $k=3$, we prove 
that $c(n,k)$ are integers and study their divisibility properties.
\end{abstract}

\maketitle

\newcommand{\nn}{\nonumber}
\newcommand{\ba}{\begin{eqnarray}}
\newcommand{\ea}{\end{eqnarray}}
\newcommand{\ift}{\int_{0}^{\infty}}
\newcommand{\ifft}{\int_{- \infty}^{\infty}}
\newcommand{\no}{\noindent}
\newcommand{\realpart}{\mathop{\rm Re}\nolimits}
\newcommand{\imagpart}{\mathop{\rm Im}\nolimits}

\newtheorem{Definition}{\bf Definition}[section]
\newtheorem{Thm}[Definition]{\bf Theorem}
\newtheorem{Example}[Definition]{\bf Example}
\newtheorem{Lem}[Definition]{\bf Lemma}
\newtheorem{Note}[Definition]{\bf Note}
\newtheorem{Cor}[Definition]{\bf Corollary}
\newtheorem{Conj}[Definition]{\bf Conjecture}
\newtheorem{Prop}[Definition]{\bf Proposition}
\newtheorem{Problem}[Definition]{\bf Problem}
\numberwithin{equation}{section}

\section{Introduction} \label{sec-intro}
\setcounter{equation}{0}

In a recent issue of the American Mathematical Monthly, 
Hugh Montgomery and Harold S. Shapiro proposed the 
following problem (Problem 11380, August-September 2008):  \\

\noindent
For $x \in \mathbb{R}$, let 
\begin{equation}
\binom{x}{n} = \frac{1}{n!} \prod_{j=0}^{n-1} (x-j). 
\label{mont-shap}
\end{equation}
\noindent
For $n \geq 1$, let $a_{n}$ be the numerator and $q_{n}$ the denominator of 
the rational number $\binom{-1/3}{n}$ expressed as a reduced fraction, with
$q_{n} > 0$.
\begin{enumerate}
 \item Show that $q_{n}$ is a power of $3$.
 \item Show that $a_{n}$ is odd if and only if $n$ is a sum of distinct powers
of $4$.
\end{enumerate}

\medskip

Our approach to this problem employs Legendre's remarkable 
expression  \cite{legendre1}:
\begin{equation}
\nu_{p}(n!) = \frac{n - s_{p}(n)}{p-1},
\label{legen}
\end{equation}
\noindent
that relates the $p$-adic valuation of factorials to the sum of digits of 
$n$ in base $p$. For $m \in \mathbb{N}$ and a prime $p$, the $p$-adic 
valuation of $m$, denoted by $\nu_{p}(m)$, is the highest power of $p$
that divides $m$. The expansion of $m \in \mathbb{N}$ in base $p$ is written
as 
\begin{equation}
m = a_{0} + a_{1}p + \cdots + a_{d}p^{d},
\end{equation}
\noindent
with integers $0 \leq a_{j} \leq p-1$ and $a_{d} \neq 0$. The function $s_{p}$
in (\ref{legen}) is defined by
\begin{equation}
s_{p}(m) := a_{0} + a_{1} + \cdots + a_{d}. 
\label{sum-def}
\end{equation}

Since, for $n > 1$, $\nu_{p}(n) = \nu_{p}(n!) - \nu_{p}((n-1)!)$, it 
follows from (\ref{legen}) that 
\begin{equation}
\nu_{p}(n) = \frac{1 + s_{p}(n-1) - s_{p}(n) }{p-1}.
\end{equation}

The $p$-adic valuations of binomial coefficients can be expressed in terms of
the function $s_{p}$: 
\begin{equation}
\nu_{p} \left( \binom{n}{k} \right) = 
\frac{s_{p}(k) + s_{p}(n-k) - s_{p}(n)}{p-1}.
\end{equation}
\noindent
In particular, for the central binomial coefficients $C_{n} := \binom{2n}{n}$
and $p=2$, we have
\begin{equation}
\nu_{2} \left( C_{n} \right) = 
2s_{2}(n)  - s_{2}(2n) = s_{2}(n).
\end{equation}
Therefore, $C_{n}$ is always even and
$\tfrac{1}{2}C_{n}$ is odd
precisely whenever $n$ is a power of $2$. This is a well-known result. \\

The central binomial coefficients $C_{n}$ have the 
generating function
\begin{equation}
(1 - 4x)^{-1/2} = \sum_{n \geq 0} C_{n}x^{n}.
\end{equation}
The binomial theorem shows that the numbers in the Montgomery-Shapiro 
problem bear a similar generating function
\begin{equation}
(1 - 9x)^{-1/3} = \sum_{n \geq 0} \binom{-\tfrac{1}{3}}{n} (-9x)^{n}.
\end{equation}

It is natural to consider the coeffients $c(n,k)$ defined by 
\begin{equation}
(1 - k^{2}x)^{-1/k} = \sum_{n \geq 0} c(n,k) x^{n},
\label{c-def}
\end{equation}
\noindent
which include the central binomial coefficients as
a special case. We call $c(n,k)$ the {\em k-central binomial coefficients}. 
The expression 
\begin{equation}
c(n,k) = (-1)^{n} \binom{- \tfrac{1}{k}}{n} k^{2n}
\end{equation}
\noindent
comes directly from the binomial theorem.  Thus, the 
Montgomery-Shapiro question from (\ref{mont-shap}) deals with 
arithmetic properties of 
\begin{equation}
\binom{-\tfrac{1}{3}}{n} = (-1)^{n} \frac{c(n,3)}{3^{2n}}.
\label{mon-sha}
\end{equation}

\section{The integrality of $c(n,k)$} \label{sec-rat}
\setcounter{equation}{0}

It is a simple matter to verify that the coefficients $c(n,k)$ are rational
numbers. The expression produced in the next proposition is then employed
to prove that $c(n,k)$ are actually integers. The 
next section will explore divisibility
properties of the integers $c(n,k)$. 

\begin{Prop}
The coefficient $c(n,k)$ is given by 
\begin{equation}
c(n,k) = \frac{k^{n}}{n!} \prod_{m=1}^{n-1}(1 + km).
\label{form-c}
\end{equation}
\end{Prop}
\begin{proof}
The binomial theorem yields
\begin{eqnarray}
(1 - k^{2}x)^{-1/k} & = & \sum_{n \geq 0} \binom{- \tfrac{1}{k}}{n} 
(-k^{2}x)^{n} 
\nonumber \\
& = & \sum_{n \geq 0} \frac{k^{n}}{n!} \left( \prod_{m=1}^{n-1}( 1 + km) 
\right) x^{n}, \nonumber 
\end{eqnarray}
\noindent
and (\ref{form-c}) has been established. 
\end{proof}

An alternative proof of the previous result  is obtained from the simple 
recurrence 
\begin{equation}
c(n+1,k) = \frac{k(1+kn)}{n+1}c(n,k), \quad \text{ for } n \geq 0,
\label{recurr-00}
\end{equation}
\noindent
and its initial condition $c(0,k) = 1$. To prove (\ref{recurr-00}), simply 
differentiate (\ref{c-def}) to produce 
\begin{equation}
k(1-k^{2}x)^{-1/k-1} = \sum_{n \geq 0} (n+1) c(n+1,k) x^{n}
\end{equation}
\noindent 
and multiply both sides by $1-k^{2}x$ to get the result.  \\

\noindent
{\bf Note}. The coefficients 
$c(n,k)$ can be written in terms of the Beta function as
\begin{equation}
c(n,k) = \frac{k^{2n}}{n B(n,1/k)}.
\label{c-beta}
\end{equation}
\noindent
This expression follows directly by writing the product in (\ref{form-c}) 
in terms of the Pochhammer symbol $(a)_{n} = a(a+1) \cdots (a+n-1)$ and the
identity
\begin{equation}
(a)_{n} = \frac{\Gamma(a+n)}{\Gamma(a)}. 
\end{equation}
\noindent
The proof employs only the most elementary properties of the Euler's Gamma 
and Beta 
functions. The reader can find details in \cite{irrbook}. The conclusion is 
that we have an integral expression for $c(n,k)$, given by 
\begin{equation}
c(n,k) \int_{0}^{1} (1 - u^{1/n})^{1/k-1} \, du = k^{2n}. 
\end{equation}
\noindent 
It is unclear how to use it to further investigate $c(n,k)$. 

\medskip

In the case $k=2$, we have that $c(n,2) = C_{n}$ is a positive integer. This 
result extends to all values of $k$. 

\begin{Thm}
\label{thm-integer}
The coefficient $c(n,k)$ is a positive integer. 
\end{Thm}

\begin{proof}
First observe that if $p$ is a prime dividing $k$, then the product in 
(\ref{c-def}) is relatively prime to $p$. Therefore we need to check that 
$\nu_{p}(n!) \leq \nu_{p}(k^{n})$. This is simple:
\begin{equation}
\nu_{p}(n!) = \frac{n - s_{p}(n)}{p-1} \leq n \leq \nu_{p}(k^{n}). 
\end{equation}

Now let $p$ be  a prime not dividing $k$. Clearly,
\begin{equation}
\nu_{p}(c(n,k)) = \nu_{p} \left( \prod_{m<n} (1+km) \right) - 
\nu_{p} \left( \prod_{m<n} (1 + m) \right). 
\end{equation}
\noindent
To prove that $c(n,k)$ is an integer, we compare the $p$-adic valuations of 
$1+km$ and $1+m$. Observe that $1+m$ is divisible by $p^{\alpha}$ if and only 
if $m$ is of the form $\lambda p^{\alpha} -1$. On the other hand, $1+km$ is
divisible by $p^{\alpha}$ precisely when $m$ is of the form 
$\lambda p^{\alpha} - i_{p^{\alpha}}(k)$, where 
$i_{p^{\alpha}}(k)$ denotes the inverse of $k$ modulo $p^{\alpha}$ in the 
range $1, \, 2, \, \cdots, p^{\alpha}-1$. Thus, 
\begin{equation}
\nu_{p}(c(n,k)) = \sum_{\alpha \geq 1} 
\Big{\lfloor{} \frac{n + i_{p^{\alpha}}(k) -1}{p^{\alpha}}  \Big{\rfloor}} - 
\Big{\lfloor{} \frac{n}{p^{\alpha}}  \Big{\rfloor}}. 
\label{val-p}
\end{equation}
\noindent
The claim now follows from $i_{p^{\alpha}}(k) \geq 1$. 
\end{proof}

Next, Theorem \ref{thm-integer} will be slightly strengthened and an
alternative proof be provided. 

\begin{Thm}
\label{thm-integer-}
For $n > 0$, the coefficient $c(n,k)$ is a positive integer divisible by $k$. 
\end{Thm}
\begin{proof}
Expanding the right hand side of the identity 
\begin{equation}
(1 - k^{2}x)^{-1}  = \left( (1 - k^{2}x)^{-1/k} \right)^{k}
\end{equation}
\noindent
by the Cauchy product formula gives 
\begin{equation}
\sum_{i_{1}+ \cdots + i_{k}=m} c(i_{1},k)c(i_{2},k) \cdots c(i_{k},k) = 
k^{2m}, 
\label{multi-sum}
\end{equation}
\noindent
where the multisum runs through all the $k$-tuples of non-negative integers. 
Obviously $c(0,k) =1$ and it is easy to check that $c(1,k) = k$. We proceed 
by induction on $n$, so we assume the assertion is valid for 
$c(1,k), \, c(2,k), \cdots, c(n-1,k)$. We prove the same is true for $c(n,k)$. 
To this end, break up (\ref{multi-sum}) as 
\begin{equation}
kc(n,k) + \sum_{\substack{i_{1}+ \cdots + i_{k}=n \\ 0 \leq i_{j} < n}} 
c(i_{1},k)c(i_{2},k) \cdots c(i_{k},k) = 
k^{2n}. 
\label{multi-sum1}
\end{equation}
\noindent
Hence by the induction assumption $kc(n,k)$ is an integer. 

To complete the proof, divide 
(\ref{multi-sum1}) through by $k^{2}$ and rewrite as follows
\begin{equation}
\frac{c(n,k)}{k} = k^{2n-2} - \frac{1}{k^{2}} 
\sum_{\substack{
i_{1}+ \cdots + i_{k}=n \\
0 \leq i_{j} < n}} 
c(i_{1},k)c(i_{2},k) \cdots c(i_{k},k).
\label{multi-sum2}
\end{equation}
\noindent
The key point is that each summand in (\ref{multi-sum2}) contains 
{\em at least two} terms, each one divisible by $k$. 
\end{proof}

\medskip

\noindent
{\bf Note}. W. Lang \cite{lang1} has studied the numbers appearing in the 
generating function
\begin{equation}
c2(l;x) := \frac{1 - ( 1 - l^{2}x)^{1/l}}{lx},
\end{equation}
\noindent
that bears close relation to the case $k = -l < 0$ of equation (\ref{c-def}).
The special case $l=2$ yields the Catalan numbers. 
The author establishes the integrality of the coefficients in the 
expansion of $c2$ and 
other related functions. \\

\section{The valuation of $c(n,k)$} \label{sec-valuation}
\setcounter{equation}{0}

We consider now the $p$-adic valuation of $c(n,k)$. The special case 
when $p$ divides
$k$ is easy, so we deal with it first.

\begin{Prop}
Let $p$ be a prime that divides $k$. Then
\begin{equation}
\nu_{p}(c(n,p)) = \nu_{p}(k) n - \frac{n-s_p(n)}{p-1}.
\end{equation}
\end{Prop}
\begin{proof}
The $p$-adic valuation of $c(n,p)$ is given by 
\begin{equation}
\nu_{p}(c(n,p)) = \nu_{p}(k) n - \nu_{p}(n!) = \nu_{p}(k) n - \frac{n-s_p(n)}{p-1}. 
\end{equation}
\noindent
Finally note that $s_{p}(n) = O(\log n)$. 
\end{proof}

\noindent
{\bf Note}. For $p, \, k \neq 2$, we have 
$\nu_{p}(c(n,p)) \sim \left(\nu_{p}(k)-\frac{1}{p-1}\right) n$, as $n \to 
\infty$. \\

We now turn attention to the case  where $p$ does not divide $k$. Under this 
assumption, the congruence $kx \equiv 1 \bmod p^{\alpha}$ has a solution. 
Elementary arguments of $p$-adic analysis can be 
used to produce a $p$-adic integer that yields the inverse of $k$. This 
construction proceeds as follows: first choose $b_{0}$ in the
range $\{ 1, \, 2, \, \cdots, p-1 \}$ to satisfy $kb_{0} \equiv 1 \bmod p$. 
Next, choose $c_{1}$, satisfying $kc_{1} \equiv 1 \bmod p^{2}$ and 
write it as $c_{1} = b_{0} + kb_{1}$ with 
$0 \leq b_{1} \leq  p-1$. Proceeding 
in this manner, we obtain a
sequence of integers $\{ b_{j}: \, j \geq 0 \}$, such that 
$ 0 \leq b_{j} \leq p-1$ and the partial sums of the {\em formal object} 
$x = b_{0} + b_{1}p + b_{2}p^{2} + \cdots $ satisfy
\begin{equation}
k \left( b_{0} + b_{1}p + \cdots + b_{j-1}p^{j-1} \right) \equiv 1 \bmod p^{j}.
\end{equation}
\noindent
This is the standard definition of a $p$-adic integer and 
\begin{equation}
i_{p^{\infty}}(k) = \sum_{j=0}^{\infty} b_{j}p^{j}
\label{inverse}
\end{equation}
\noindent
is the inverse of $k$ in the ring of $p$-adic integers. The reader will find in 
\cite{gouvea1} and \cite{murty3} information about this topic.  \\

\noindent
{\bf Note}. It is convenient to modify the notation in (\ref{inverse}) and 
write it as 
\begin{equation}
i_{p^{\infty}}(k) = 1+ \sum_{j=0}^{\infty} b_{j}p^{j}
\label{inverse-1}
\end{equation}
\noindent
which is always possible 
since the first coefficient cannot be zero. The reader is
invited to check that, when doing so, the $b_j$ are periodic in $j$ with 
period the
multiplicative order of $p$ in $\mathbb{Z}/k\mathbb{Z}$. Furthermore, 
the $b_j$ take values amongst 
${\lfloor p/k \rfloor, \lfloor 2p/k \rfloor, \ldots, \lfloor (k-1)p/k \rfloor}$.
This will be exemplified in the case $k=3$ later. \\

The analysis of $\nu_{p}(c(n,k))$ for those primes $p$ not dividing $k$ begins
with a characterization of those indices for which $\nu_{p}(c(n,k)) = 0$, that
is, $p$ does not divide $c(n,k)$. The result is expressed in terms of the 
expansions of $n$ in base $p$, written as 
\begin{equation}
n = a_{0} + a_{1}p + a_{2}p^{2} + \cdots + a_{d}p^{d},
\end{equation}
\noindent
and the $p$-adic expansion of the inverse of $k$ as given by (\ref{inverse-1}).

\begin{Thm}
\label{thm-main}
Let $p$ be a prime that does not divide $k$. Then $\nu_{p}(c(n,k)) = 0$ 
if and only if $a_{j}+b_{j} < p$ for all $j$ in the range $1 \leq j \leq d$.
\end{Thm}
\begin{proof}
It follows from (\ref{val-p}) that $c(n,k)$ is not divisible by $p$ 
precisely when 
\begin{equation}
\Big{\lfloor{} \frac{1}{p^{\alpha}} \left( 
n + \sum_{j} b_{j}p^{j} \right) 
\Big{\rfloor}} = \Big{\lfloor{} \frac{n}{p^{\alpha}} \Big{\rfloor}},
\end{equation}
\noindent
for all $\alpha \geq 1$, or equivalently, if and only if 
\begin{equation}
\sum_{j=0}^{\alpha-1} (a_{j}+b_{j})p^{j} < p^{\alpha},
\end{equation}
\noindent
for all $\alpha \geq 1$. An inductive argument shows that this is 
equivalent to the condition $a_{j} + b_{j} < p$ for all $j$. Naturally, the 
$a_{j}$ vanish for $j > d$, so it is sufficient to check $a_{j}+b_{j} < p$
for all $j \leq d$. 
\end{proof}

\begin{Cor}
For all primes $p > k$ and $d \in \mathbb{N}$, 
we have $\nu_{p}(c(p^{d},k)) = 0$. 
\end{Cor}
\begin{proof}
The coefficients of $n = p^{d}$ in Theorem \ref{thm-main} are $a_{j}=0$ 
for $0 \leq j \leq d-1$ and $a_{d}=1$. Therefore the restrictions on the
coefficients $b_{j}$ become 
$b_{j} < p$ for $0 \leq j \leq d-1$ and $b_{d} < p-1$. It turns out that 
$b_{j} \neq p-1$ for all $j \in \mathbb{N}$. Otherwise, for some $r \in 
\mathbb{N}$, we have $b_{r} = p-1$ and the equation 
\begin{equation}
k \left( 1 + \sum_{j=0}^{r-1}b_{j}p^{j} + b_{r}p^{r} \right) 
\equiv 
k \left( 1 + \sum_{j=0}^{r-1}b_{j}p^{j} - p^{r} \right) 
\equiv 1 \bmod p^{r+1}, 
\end{equation}
\noindent
is impossible in view of 
\begin{equation}
-kp^{r}  < k \left( 1 + \sum_{j=0}^{r-1} b_{j}p^{j} - p^{r} \right) < 0. 
\end{equation}
\end{proof}

\medskip

Now we return again to the Montgomery-Shapiro question. The identity 
(\ref{mon-sha})
shows that the denominator $q_{n}$ is a power of $3$. We now consider 
the indices $n$ for which $c(n,3)$ is odd and provide a proof of the 
second part
of their problem.

\begin{Cor}
The coefficient $c(n,3)$ is odd precisely when $n$ is a sum of distinct 
powers of $4$. 
\end{Cor}
\begin{proof}
The result follows from Theorem \ref{thm-main} and the explicit formula
\begin{equation}
i_{2^{\infty}}(3) = 1 + \sum_{j=0}^{\infty} 2^{2j+1},
\end{equation}
\noindent
for the inverse of $3$, so that $b_{2j}=0$ and $b_{2j+1}=1$. Therefore, if 
$c(n,3)$ is odd, the theorem 
now shows that $a_{j}=0$ for $j$ odd, as claimed. 
\end{proof}

More generally, the discussion of $\nu_{p}(c(n,3)) = 0$ is divided according 
to the residue of $p$ modulo $3$. This division is a consequence of the fact
that for $p = 3u+1$, we have 
\begin{equation}
i_{p^{\infty}}(3) = 1 + 2u \sum_{m=0}^{\infty} p^{m},
\end{equation}
\noindent 
and for $p= 3u+2$, one computes $p^{2} = 3(3u^{2}+4u+1)+1$, to conclude that
\begin{equation}
i_{p^{\infty}}(3) = 1 + 2(3u^{2}+4u+1) \sum_{m=0}^{\infty} p^{2m} = 
1 + \sum_{m=0}^{\infty} up^{2m} + (2u+1)p^{2m+1}. 
\end{equation}

\begin{Thm}
Let $p \neq 3$ be a prime and $n = a_{0}+a_{1}p + a_{2}p^{2} + \ldots + 
a_{d}p^{d}$ as before. Then $p$ does not divide $c(n,3)$ if and only if
the $p$-adic digits of $n$ satisfy
\begin{equation}
a_{j} < \begin{cases} 
p/3 \quad & \text{if $j$ is odd or $p = 3u+1$,} \\
2p/3 \quad & \text{otherwise.}
\end{cases}
\end{equation}
\end{Thm}

For general $k$ we have the following analogous statement.

\begin{Thm} \label{thm-p1nodivision}
 Let $p=ku+1$ be a prime. Then $p$ does not divide $c(n,k)$ if and only if the
 $p$-adic digits of $n$ are less than $p/k$.
\end{Thm}

Observe that Theorem \ref{thm-p1nodivision} implies the following well-known
property of the central binomial coefficients: $C_n$ is not divisible by
$p \neq 2$ if and only if
the $p$-adic digits of $n$ are less than $p/2$.

\medskip

Now we return to (\ref{val-p}) which will be written as
\begin{equation}
\nu_{p}(c(n,k)) = \sum_{\alpha \geq 0} \Big{\lfloor{} \frac{1}{p^{\alpha+1}}
\, \sum_{m=0}^{\alpha} \,  (a_{m}+b_{m})p^{m} \Big{\rfloor}}. 
\label{val-p2}
\end{equation}
\noindent
From here, we bound 
\begin{equation}
\sum_{m=0}^{\alpha} (a_{m}+b_{m})p^{m} \leq 
\sum_{m=0}^{\alpha} (2p-2)p^{m} = 2(p^{\alpha+1}-1) < 2p^{\alpha+1}.
\end{equation}
\noindent
Therefore, each summand in (\ref{val-p2}) is either $0$ or $1$. The $p$-adic 
valuation of $c(n,p)$ counts the number of $1$'s in this sum. This proves the 
final result.

\begin{Thm}
Let $p$ be a prime that does not divide $k$. Then, with the previous notation 
for $a_{m}$ and $b_{m}$, we have that $\nu_{p}(c(n,k))$ is the number of 
indices $m$ such that either
\begin{itemize}
 \item $a_m + b_m \geq p$ or
 \item there is $j \leq m$ such that $a_{m-i}+b_{m-i} = p-1$ for $0 \leq i 
\leq j-1$ and $a_{m-j}+b_{m-j} \geq p$.
\end{itemize}
\end{Thm}

\begin{Cor}
Let $p$ be a prime that does not divide $k$, and write $n = \sum a_{m}p^{m}$ and
$i_{p^{\infty}}(k) = 1 + \sum b_{m}p^{m}$, as before. Let $v_{1}$ and $v_{2}$
be the number of indices $m$ such that $a_{m} + b_{m} \geq p$ and 
$a_{m} + b_{m} \geq p-1$, respectively. Then 
\begin{equation}
v_{1} \leq \nu_{p}( c(n,k) ) \leq v_{2}.
\end{equation}
\end{Cor}

\medskip

\section{A $q$-generalization of $c(n,k)$} \label{sec-q}
\setcounter{equation}{0}

A standard procedure to generalize an integer expression is to replace 
$n \in \mathbb{N}$ by the polynomial 
\begin{equation}
[q]_{n} := \frac{1-q^{n}}{1-q} = 1+q+q^2+\ldots+q^{n-1}.
\end{equation}
\noindent
The original expression is recovered as the limiting case $q \to 1$. For 
example, the factorial $n!$ is extended to the polynomial
\begin{equation}
[n]_{q}! := [n]_q [n-1]_q \dots [2]_q [1]_q = \prod_{j=1}^{n} \frac{1-q^{j}}{1-q}. 
\end{equation}
\noindent
The reader will find in \cite{kac-chung} an introduction to this $q$-world. \\

In this spirit we generalize the integers 
\begin{equation}
c(n,k) = \frac{k^{n}}{n!} \prod_{m=0}^{n-1}(km+1) = \prod_{m=1}^{n}\frac{k(k(m-1)+1)}{m},
\end{equation}
\noindent
into the $q$-world as
\begin{equation}
F_{n,k}(q) := \prod_{m=1}^{n}\frac{[km]_q [k(m-1)+1]_q}{[m]_q^2}.
\end{equation}
\noindent
Note that this expression indeed gives $c(n,k)$ as $q \to 1$.
The corresponding extension of Theorem \ref{thm-integer} is stated in the 
next result. The proof is similar to that given above, so it is left to the
curious reader. 

\begin{Thm}
\label{Fn-poly}
The function
\begin{equation}
F_{n,k}(q) := \prod_{m=1}^{n}\frac{(1-q^{km}) (1-q^{k(m-1)+1})}{(1-q^m)^2}
\end{equation}
is a polynomial in $q$ with integer coefficients. 
\end{Thm}

\medskip

\section{Future directions} \label{sec-future}
\setcounter{equation}{0}

In this final section we discuss some questions related to the integers 
$c(n,k)$. \\

\noindent
$\bullet$ {\bf A combinatorial interpretation}. The integers $c(n,2)$ are
given by the central binomial coefficients $C_{n} = \binom{2n}{n}$. These 
coefficients
appear in many counting situations: $C_{n}$ gives the number of walks of 
length
$2n$ on an infinite linear lattice that begin and end at the origin. Moreover, 
they provide the exact answer for the elementary sum
\begin{equation}
\sum_{k=0}^{n} \binom{n}{k}^{2} = C_{n}. 
\end{equation}
\noindent
Is it possible to produce similar results for $c(n,k)$, with $k \neq 2$? In 
particular, what do the numbers $c(n,k)$ count? \\

\noindent
$\bullet$ {\bf A further generalization}. The 
polynomial $F_{n,k}(q)$ can be written as 
\begin{equation}
F_{n,k}(q) = \frac{(1-q)}{(1-q^{kn+1})} \prod_{m=1}^{n}
\frac{(1 - q^{km}) (1 - q^{km+1}) }{(1 - q^{m})^{2}}
\end{equation}
\noindent
which suggests the extension
\begin{equation}
G_{n,k}(q,t) := \frac{(1-q)}{(1-tq^{kn})} \prod_{m=1}^{n}
\frac{(1 - q^{km}) (1 - tq^{km}) }{(1 - q^{m})^{2}}
\end{equation}
\noindent
so that $F_{n,k}(q) = G_{n,k}(q,q)$. Observe that $G_{n,k}(q,t)$ is not always
a  polynomial. For example,
\begin{equation}
G_{2,1}(q,t) = \frac{1- qt}{1-q^{2}}.
\end{equation}
\noindent
On the other hand, 
\begin{equation}
G_{1,2}(q,t) = q+1.
\end{equation}

The following functional equation is easy to establish. 

\begin{Prop}
The function $G_{n,k}(q,t)$ satisfies 
\begin{equation}
G_{n,k}(q,tq^{k}) = \frac{(1- q^{kn}t)}{(1-q^{k}t)}G_{n,k}(q,t).
\end{equation}
\end{Prop}

The reader is invited to explore its properties. In particular, find minimal 
conditions on $n$ and $k$ to guarantee that $G_{n,k}(q,t)$ is a polynomial.

Consider now the function
\begin{equation}
H_{n,k,j}(q) := G_{n,k}(q,q^{j})
\end{equation}
\noindent
that extends $F_{n,k}(q) = H_{n,k,1}(q)$. The following statement predicts the 
situation where $H_{n,k,j}(q)$ is a polynomial. \\

\noindent
{\bf Problem}. Show that $H_{n,k,j}(q)$ is a polynomial precisely if the 
indices  satisfy $k \equiv 0 \bmod \text{gcd}(n,j)$. 

\medskip

\noindent
$\bullet$ {\bf A result of Erd\"os, Graham, Ruzsa and Strauss}. In this 
paper we have explored the 
conditions on $n$ that result in $\nu_{p}(c(n,k)) = 0$. Given two 
distinct primes $p$ and $q$, P. Erd\"os et al. 
\cite{paul2} discuss the existence of indices $n$ for which 
$\nu_{p}(C_{n}) = \nu_{q}(C_{n}) = 0$. Recall that 
by Theorem \ref{thm-p1nodivision}
such numbers $n$ are characterized by having $p$-adic digits less than $p/2$ and
$q$-adic digits less than $q/2$. The following result of \cite{paul2}
proves the existence of infinitely many such $n$.

\begin{Thm}
Let $A, \, B \in \mathbb{N}$ such that $A/(p-1) + B/(q-1) \geq 1$. Then
there exist infinitely many numbers $n$ with $p$-adic digits $\leq A$ and
$q$-adic digits $\leq B$. 
\end{Thm}

This leaves open the question for $k>2$ whether or not there exist 
infinitely many
numbers $n$ such that $c(n,k)$ is neither divisible
by $p$ nor by $q$. The extension 
to more than two primes is open even in the case $k=2$. In particular, a prize
of $\$ 1000$ has been offered by R. Graham for just showing that there are 
infinitely many $n$ such
that $C_{n}$ is coprime to  $105 = 3 \cdot 5 \cdot 7$. 
On the other hand, it is conjectured that there are only finitely many 
indices $n$ such that $C_{n}$ is not divisible 
by any of $3, \, 5, \, 7$ and $11$.  \\

Finally, we remark that 
Erd\"os et al. conjectured in \cite{paul2} that the central
binomial coefficients $C_n$ are never squarefree for $n>4$ which has been
proved by Granville and Ramare in \cite{granville2}. Define 
\begin{equation}
\tilde{c}(n,k) := \text{Numerator} \left( k^{-n} c(n,k) \right).
\end{equation}
\noindent
We have {\em some} empirical evidence which suggests the existence 
of an index $n_{0}(k)$, such that 
$\tilde{c}(n,k)$ is not squarefree for $n \geq n_{0}(k)$. The value 
of $n_{0}(k)$ could be large. For instance 
\begin{eqnarray}
\tilde{c}(178,5) & = & 10233168474238806048538224953529562250076040177895261 \nonumber \\
& & 58561031939088200683714293748693318575050979745244814 \nonumber \\
& & 765545543340634517536617935393944411414694781142
\nonumber
\end{eqnarray}
\noindent
is squarefree, so that $n_{0}(5) \geq 178$.  The numbers $\tilde{c}(n,k)$ 
present new challeges, even in the  case $k=2$. Recall that $\tfrac{1}{2}C_{n}$
is odd if and only if $n$ is a power of $2$. Therefore, $C_{786}$ is not 
squarefree. On the other hand, the  
complete factorization of $C_{786}$ shows that 
$\tilde{c}(786,2)$ is squarefree. We conclude that $n_{0}(2) \geq 786$.

\medskip

\no
{\bf Acknowledgments}. The work of the second author was partially funded by
$\text{NSF-DMS } 0409968$. The first author was partially supported, as 
a graduate student, by the same grant.

\end{document}